\documentclass{amsart}

\newtheorem{theorem}{Theorem}[section]
\newtheorem{definition}[theorem]{Definition}
\newtheorem{lemma}[theorem]{Lemma}
\newtheorem{corollary}[theorem]{Corollary}
\newtheorem{proposition}[theorem]{Proposition}
\newtheorem{question}[theorem]{Question}
\usepackage{tikz}
\usepackage{color}
\usetikzlibrary{graphs,positioning}

\renewcommand{\d}{\vec}

\title{Infinite graphs with finite metric dimension}

\author{Csaba Bir\'o}
\address{Department of Mathematics, University of Louisville, Louisville, KY 40292}
\email{csaba.biro@louisville.edu}

\author{Caroline E. Boone}
\address{Department of Mathematics, University of Louisville, Louisville, KY 40292}
\email{caroline.boone@louisville.edu}

\author{Beth Novick}
\address{School of Mathematical and Statistical Sciences, Clemson University, Clemson, SC 29634}
\email{nbeth@clemson.edu}

\author{Hazel Torek}
\address{School of Computing, Clemson University, Clemson, SC 29634}
\email{ctorek@clemson.edu}

\begin{document}

\begin{abstract}
We study the metric dimension (strong and weak) of connected infinite graphs. In particular, our main interest is characterizing infinite graphs with finite dimension.  Our main results:
(1) graphs with more than one end have infinite strong dimension;
(2) for locally finite graphs with a finite number of cycles, the weak dimension   is finite if and only if the graph  has finitely many vertices of degree at least three, and 
the strong dimension  is finite if and only if the graph has one end and finitely many vertices of degree at least three.

\end{abstract}

\maketitle

\section{Introduction}

We begin by introducing relevant definitions and background material.  We consider both finite and infinite graphs,  our main focus being infinite graphs that are locally finite and  connected.  See \cite{diestel} for graph theory terminology. 

\subsection{The metric dimension of a graph}
Let $G$ be a connected graph, and let $W$ be a set of its vertices. We call $W$ a \emph{(weak) resolving set}, if for all $u,v\in V(G)$, $u\neq v$, there is a $w\in W$ such that $d(u,w)\neq d(v,w)$. We say that $w$ \emph{(weakly) resolves} $u$ and $v$. The smallest cardinality of a weak resolving set is called the \emph{(weak) metric dimension} of $G$, denoted by $\beta(G)$. Note that in the case of infinite graphs, $\beta(G)$ may be infinite, or even uncountable.

Metric dimension for finite graphs was introduced independently by Harary and Melter \cite{Harary} and Slater \cite{Slater}.  Since then,  the parameter has been studied extensively and applied in diverse areas.  The problem of determining $\beta(G)$ for an arbitrary connected graph is known to be NP-hard,  see \cite{Khuller}.  For papers that discuss applications and complexity,  see for example those cited in the publications of C{\'a}ceres et al.  \cite{Caceres1} and Belmonte et al. \cite{Belmonte}.  

Seb\H{o} and Tannier \cite{Sebo} observed that it is possible for two non-isomorphic graphs with the same vertex set to have the same list of distances to some resolving set.  This observation motivated them to introduce a stronger parameter:  
We call $W$ a \emph{strong resolving set}, if for all $u,v\in V(G)$, $u\neq v$, there is a $w\in W$ such that a shortest $u$--$w$ path contains $v$ or a shortest $v$--$w$ path contains $u$. We say that $w$ \emph{strongly resolves} $u$ and $v$. The smallest cardinality of a strong resolving set is called the \emph{strong metric dimension} of $G$, denoted by $\beta_S(G)$.
The problem of determining $\beta_S(G)$ for an arbitrary finite graph is also NP-hard, as was shown by Oellermann and Peters-Fransen, who gave a polynomial-time equivalence to a certain vertex covering problem \cite{Peters-Fransen}.  See \cite{kratica} for a survey on strong metric dimension of finite graphs.  Also see \cite{Benakli} for a survey including references for weak and strong metric dimension. 

The study of metric dimension of infinite graphs was initiated by C{\'a}ceres et al. \cite{Caceres} and the body of literature on that topic is much smaller.  To the best of our knowledge,  the current article is the first to consider strong metric dimension of infinite graphs.

Clearly, if $w$ strongly resolves $u$ and $v$, then it also weakly resolves them, so a strong resolving set is always a weakly resolving set. This shows $\beta_S(G)\geq\beta(G)$.

In this paper, for sake of brevity, we will usually drop the word \emph{metric} from these names, and just use \emph{weak dimension} and \emph{strong dimension}. If a proof only involves one of these, we will often just use the word \emph{dimension}. Similar rules apply for the terms \emph{resolving}, and \emph{resolving set}. However, the notations $\beta$ and $\beta_S$ will always be used correctly.

If $G$ is not connected, the terms above are not defined. Therefore we will always assume that the graph of our interest is connected, and use this fact in our arguments freely.

Graphs that are not locally finite have infinite weak (and hence strong) dimension. This is due to the following theorem \cite{Caceres}.

\begin{theorem}
If $\beta(G)=k$, then $G$ has maximum degree at most $3^k-1$.
\end{theorem}

Since our interest is mainly to study infinite graphs with finite metric dimension, we will restrict our attention to locally finite graphs. We will make this assumption automatically in the rest of the paper, unless otherwise noted. Note that locally finite graphs are necessarily countable, and the dimension will always be countable. So one could replace the symbol $\infty$ with $\omega$ throughout the paper. We decided to use $\infty$ instead of $\omega$ to emphasize the combinatorial (as opposed to set theoretical) nature of our results.

\section{Additional definitions}

A \emph{ray} is a graph $(V,E)$ of the form $V=\{v_0,v_1,\ldots\}$, and $E=\{v_0v_1,v_1v_2,\ldots\}$. A subray of a ray is called a \emph{tail} of that ray. Following \cite{diestel}, we will often refer to a ray by the natural sequence of its vertices and write $v_0v_1\ldots$ for this ray.

A \emph{double ray} is a graph $(V,E)$ of the form $V=\{v_i:i\in\mathbb{Z}\}$, and $E=\{v_iv_{i+1}:i\in\mathbb{Z}\}$. Removal of any vertex from a double ray decomposes it into two rays. 

An \emph{end} is an equivalence class of rays, where two rays are equivalent if they cannot be separated by the removal of finitely many vertices. In other words, for a graph $G$, two rays $R_1$ and $R_2$ are equivalent, if for every finite set $S\subseteq V(G)$, the graph $G-S$ contains a tail of both $R_1$ and $R_2$ in the same component.

The following basic proposition is an important property of locally finite graphs. The proof can be found in \cite{diestel} (Proposition 8.2.1).
\begin{proposition}
Every locally finite infinite graph contains a ray, hence it has at least one end.
\end{proposition} 

We define two graphs, called the \emph{infinite ladder} (Figure~\ref{fig:ladder}), and the \emph{infinite broken ladder} (Figure~\ref{fig:brokenladder}). Both consist of two disjoint rays $a_0a_1\ldots$ and $b_0b_1\ldots$. For the infinite ladder, add edges $a_ib_i$ for all $i\geq 0$. For the infinite broken ladder, add edges $a_ib_i$ for all even $i\geq 0$.

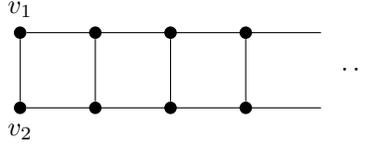
\begin{figure}
    \centering
    \begin{tikzpicture}
        \node[circle, fill=black,scale = 0.5,label = above:$a_0$] (A) at (0,0) {};
        \node[circle, fill=black,scale = 0.5] (B) at (1,0) {};
        \node[circle, fill=black,scale = 0.5] (C) at (2,0) {};
        \node[circle, fill=black,scale = 0.5,label = below:$b_0$] (D) at (0,-1) {};
        \node[circle, fill=black,scale = 0.5] (E) at (1,-1) {};
        \node[circle, fill=black,scale = 0.5] (F) at (2,-1) {};
        \node[circle, fill=black,scale = 0.5] (H) at (3,0) {};
        \node[circle, fill=black,scale = 0.5] (I) at (3,-1) {};

        \node (G) at (4.5,-.5) {$\cdots$};

        \draw (A) -- (4,0);
        \draw (A) -- (D);
        \draw (D) -- (4,-1);
        \draw (C) -- (F);
        \draw (B) -- (E);
        \draw (H) -- (I);
    \end{tikzpicture}
    \caption{Infinite Ladder}
    \label{fig:ladder}
\end{figure}

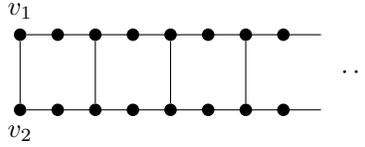
\begin{figure}
    \centering
    \begin{tikzpicture}
        \node[circle, fill=black,scale = 0.5,label = above:$a_0$] (A) at (0,0) {};
        \node[circle, fill=black,scale = 0.5] (B) at (1,0) {};
        \node[circle, fill=black,scale = 0.5] (C) at (2,0) {};
        \node[circle, fill=black,scale = 0.5,label = below:$b_0$] (D) at (0,-1) {};
        \node[circle, fill=black,scale = 0.5] (E) at (1,-1) {};
        \node[circle, fill=black,scale = 0.5] (F) at (2,-1) {};
        \node[circle, fill=black,scale = 0.5] (H) at (3,0) {};
        \node[circle, fill=black,scale = 0.5] (I) at (3,-1) {};

        \node (G) at (4.5,-.5) {$\cdots$};

        \draw (A) -- (4,0);
        \draw (A) -- (D);
        \draw (D) -- (4,-1);
        \draw (C) -- (F);
        \draw (B) -- (E);
        \draw (H) -- (I);

        \foreach \i in {1,...,4}
        {
        \node[circle, fill=black,scale = 0.5] at (\i-0.5,0) {};
        \node[circle, fill=black,scale = 0.5] at (\i-0.5,-1) {};
        }
    \end{tikzpicture}
    \caption{Infinite Broken Ladder}
    \label{fig:brokenladder}
\end{figure}

The infinite ladder has  exactly one end.  It has metric dimension $2$, with resolving set $\{a_0,b_0\}$, as depicted in Figure~\ref{fig:ladder}. 
Note that the infinite broken ladder also has metric dimension $2$,  with the same resolving set. 

We will show that infinite graphs having more than one end have infinite strong dimension.  The example of the infinite broken ladder intrigued us, being a counterexample to the converse statement:  it has only one end,  but has infinite strong dimension.  To see this, assume $W$ is a finite resolving set. There is a greatest index $i$ that appears in $W$; if $j$ is an odd integer with $j>i$, then $a_j$ and $b_j$ are not strongly resolved by $W$.

The example of the broken ladder motivated us to try to characterize graphs that have finite metric dimension and infinite strong metric dimension.

\section{Graphs with multiple ends}

In the quest to characterize infinite graphs with finite dimension, it is useful to find large classes of graphs that are definitely not candidates. In this section, we show that graphs with more than one end have infinite strong dimension. In fact, we will prove a slightly stronger statement, deducing this theorem as a corollary.

We will use the usual notion of distance of a vertex from a set of vertices: if $W$ is a set of vertices and $v$ is a vertex, then $d(v,W)=\min\{d(w,v):w\in W\}$. If $R$
is a ray, we will write $d(v,R)$ for $d(v,V(R))$. Note that in locally finite graphs, the set of vertices attaining this minimum is finite.

\begin{lemma}
Let $G$ be a graph, and $W$ be a finite set of vertices. Let $S=\{v_1,v_2,\ldots\}\subseteq V(G)$. Then
\[
\lim_{i\to\infty} d(v_i,W)=\infty.
\]
\end{lemma}

\begin{proof}
Let $W=\{w_1,\ldots,w_k\}$. Let $j_0\in[k]$, and let $d_i=d_{i,j_0}=d(v_i,w_{j_0})$. Note that the sequence $\{d_i\}$ has the property that for every nonnegative integer $m$, the set $\{i:d_i=m\}$ is finite. As a consequence of local finiteness, the set of all vertices of distance $m$ from $w_{j_0}$ is finite, which implies that $d_i\to\infty$.

Then
\[
\lim_{i\to\infty} d(v_i,W)=\lim_{i\to\infty}\min_{j\in[k]}d(v_i,w_j)=\min_{j\in[k]}\lim_{i\to\infty}d_{i,j}=\infty.
\]
\end{proof}

With this lemma, we are ready to prove the main theorem of this section.

\begin{theorem}\label{thm:3timer}
Let $G$ be a graph. If there are sets $S_1,S_2\subseteq V(G)$ such that $S_1$ and $S_2$ are infinite, but can be separated by a finite set of vertices, then $\beta_S(G)=\infty$.
\end{theorem}

\begin{proof}
Suppose that $G$ is a graph, $S_1$ and $S_2$ are infinite sets of vertices that are separated by the finite vertex set $W_0$, and suppose for a contradiction that $W_1$ is a finite resolving set. Let $W=W_0\cup W_1$. Clearly, $W$ is a finite resolving set. Let $d(v)=d(v,W)$ for $v\in V(G)$.

Let $m=\max\{d(w,w'):w,w'\in W\}$. By the Lemma, there exist $u\in S_1$ and $v\in S_2$ such that $d(u),d(v)>m$. We claim that $W$ does not resolve $u,v$.

Suppose there is $w\in W$ that resolves $u,v$; without loss of generality, a $u$--$w$ shortest path contains $v$.

First we will prove that $d(u,v)\geq d(u)+d(v)$. To see this, let $P$ be a shortest $u$--$v$ path. Notice that $P$ contains a vertex $w_0\in W$. Then
\[
d(u,v)=d(u,w_0)+d(w_0,v)\geq d(u)+d(v).
\]
Now let $w'\in W$ be such that $d(u,w')=d(u)$ (note that $w'=w_0$ is possible). Then,
\[
d(u,w)\leq d(u,w')+d(w',w)\leq d(u)+m<d(u)+d(v)\leq d(u,v),
\]
contradicting the statement that $v$ is on a shortest $u$--$w$ path.
\end{proof}

\begin{corollary}\label{cor:3timer}
If a graph $G$ has at least two ends, then $\beta_S(G)=\infty$.
\end{corollary}

We note that the main theorem is not much stronger than the corollary, as the corollary could be used to prove the theorem. To see this, consider the infinite sets $S_1$ and $S_2$ that can be separated by a finite set of vertices. After removing the separator, the graph falls into components, at least two of which are infinite. They both contain a ray and these are not in the same end of the original graph, implying the theorem.

Also note that Theorem~\ref{thm:3timer} and Corollary~\ref{cor:3timer} fail to hold for weak metric dimension. A counterexample is the $k$-spider: the union of $k$ disjoint rays, with $k\ge 3$, having a common initial vertex,  whose metric dimension is $k-1$.  See Figure~\ref{fig:kspider}.

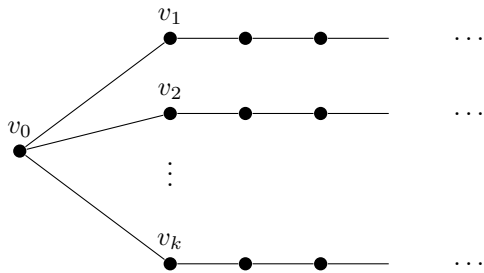
\begin{figure}
\begin{tikzpicture}[on grid,every node/.style={circle, fill, minimum size=5pt, inner sep=0pt}]
\node (C) [label=$v_0$] {};
\node (A1) [right=of C, xshift=1cm, yshift=1.5cm, label=$v_1$] {}; 
\node (A2) [right=of A1] {};
\node (A3) [right=of A2] {};
\node (A4) [right=of A3, fill=none, draw=none] {}; 
\node [right=of A4, fill=none, draw=none] {\ldots}; 
\node (B1) [right=of C, xshift=1cm, yshift=0.5cm, label=$v_2$] {};
\node (B2) [right=of B1,] {};
\node (B3) [right=of B2,] {};
\node (B4) [right=of B3, fill=none, draw=none] {}; 
\node [right=of B4, fill=none, draw=none] {\ldots}; 
\node [right=of C, xshift=1cm, yshift=-0.2cm, fill=none, draw=none] {\vdots}; 
\node (C1) [right=of C, xshift=1cm, yshift=-1.5cm, label=$v_k$] {}; 
\node (C2) [right=of C1] {}; 
\node (C3) [right=of C2] {};
\node (C4) [right=of C3, fill=none, draw=none] {};  
\node [right=of C4, fill=none, draw=none] {\ldots};
\draw (C) -- (A1) -- (A2) -- (A3) -- (A4); 
\draw (C) -- (B1) -- (B2) -- (B3) -- (B4); 
\draw (C) -- (C1) -- (C2) -- (C3) -- (C4); 
\end{tikzpicture}
\caption{$k$-spider}
\label{fig:kspider}
\end{figure}

\section{Finicyclic graphs}

In this section we characterize the class of infinite trees with finite strong dimension by providing a stronger characterization for graphs with finitely many cycles, which we call \emph{finicyclic}.

Some of the special cases discussed here were proven by C\'aceres et al.~\cite{Caceres}. We point out their results, but this paper is self-contained.

\begin{lemma}\label{lem:cyclecreator}
Let $G$ be a graph, $a,b\in V(G)$, and let $P$ be an $a$--$b$ path. Let $v$ be an internal vertex $v\in V(P)$, and let $v^-$, $v^+$ be the vertices directly preceding and directly following $v$ on $P$, respectively. If there is
an $a$--$b$ path $Q$ such that $v\not\in V(Q)$, then there is a cycle in $G$ that contains $v$,$v^-$, and $v^+$.
\end{lemma}
\begin{figure}
    \centering
    \begin{tikzpicture}
        \node[circle, fill=black,scale = 0.5,label = left:$a$] (A) at (0,0) {};
        \node[circle, fill=black,scale = 0.5,label = right:$b$] (B) at (8,0) {};
        \node[circle, fill=black,scale = 0.5,label = above:$v^-$] (C) at (3,0) {};
        \node[circle, fill=black,scale = 0.5,label = above:$v$] (D) at (4,0) {};
        \node[circle, fill=black,scale = 0.5,label = above:$v^+$] (E) at (5,0) {};
        \node[circle, fill=black,scale = 0.5,label = above:$u^-$] (H) at (1,0) {};
        \node[circle, fill=black,scale = 0.5] at (1.66,0) {};
        \node[circle, fill=black,scale = 0.5] at (2.33,0) {};
        \node[circle, fill=black,scale = 0.5] at (0.5,0) {};
        \node[circle, fill=black,scale = 0.5,label = above:$u^+$] (I) at (6,0) {};
        \node[circle, fill=black,scale = 0.5] at (6.66,0) {};
        \node[circle, fill=black,scale = 0.5] at (7.33,0) {};

        \draw (A) to (B);

        \draw (A) to [out = -50, in=-130] (H);
        \draw (H) to [out = -25, in=-155] (I);
        \draw (I) to [out = -40, in=-140] (B);
    \end{tikzpicture}
    \caption{$P$ and $Q$}
    \label{fig:pathp}
\end{figure}
\begin{proof}
Let $P$ be and $Q$ be as stated.  Let $u^-$ be the nearest vertex preceding $v$ on $P$ that is also on $Q$, and let $u^+$ be the nearest vertex following $v$ that is also on $Q$. Note that $u^-\in \{a, v^-\}$ or $u^+\in \{b, v^+\}$ are possible. Let $C$ be
\[
u^-Pu^+Qu^-.
\]
Note  that $C$ contains no repeated vertex, and hence is a cycle, which contains $v^-$, $v$, $v^+$ in this order on the $u^-Pu^+$ portion.
\end{proof}

\subsection{Characterization of locally finite finicyclic graphs}

A finicyclic graph can be made into a forest by the removal of finitely many vertices, since only one vertex must be removed from each cycle. The converse, that a graph is finicyclic provided the removal of some finite set of vertices renders it a forest, is also true but more difficult to prove since it is possible to construct a graph with a single vertex shared by infinitely many cycles. In fact, even if such a graph is locally finite, it is possible that a vertex can be contained in infinitely many cycles;  the infinite ladder has this property, for example.

However, as we shall see in the proof of the next theorem, in a finicyclic graph, it is still possible to find a finite set that destroys all cycles. Aharoni and Berger \cite{AhBer} called these type of statements ``complementary slackness conditions''.

This theorem is not essential for the second part of the section, but we found it independently interesting.

\begin{theorem}\label{thm:char}
Let $G$ be a locally finite graph. $G$ is finicyclic if and only if there is a finite set of vertices $U$ such that $G-U$ is a forest.
\end{theorem}

Before the proof of the theorem, we will prove a lemma and state a corollary.

\begin{lemma}\label{lem:infp-infc}
Let $G$ be a graph, and $a,b\in V(G)$. If there are infinitely many $a$--$b$ paths, then there are infinitely many cycles in $G$.
\end{lemma}

\begin{proof}Let $H$ be the subgraph of $G$ induced by the vertices of the set of $a$--$b$ paths. We note that $|V(H)|=\infty$. Suppose that there only finitely many cycles in $H$. Then the set of vertices $U$ that appear in a cycle of $H$ is finite. Let $W$ be the set of vertices that are on every $a$--$b$ path, and note that it is also a finite set. Hence there is a vertex $v\in V(H)$ such that $v\not\in U\cup W$.

Let $P$ be a $a$--$b$ path that contains $v$ and let $Q$ be an $a$--$b$ path that does not contain $v$. By Lemma~\ref{lem:cyclecreator}, there is a cycle in $H$ that contains $v$, a contradiction.
\end{proof}

\begin{corollary}\label{cor:finpath}
If $G$ is a finicyclic graph, then for every pair of vertices $a,b\in V(G)$, there are finitely many $a$--$b$ paths.
\end{corollary}

We note that Lemma~\ref{lem:cyclecreator}, Lemma~\ref{lem:infp-infc}, and Corollary~\ref{cor:finpath} do not require local finiteness.

We are ready to prove the characterization theorem.

\begin{proof}[Proof of Theorem~\ref{thm:char}]
It is sufficient to show that if $G-v$ is finicyclic for $v\in V(G)$, then $G$ is finicyclic. Indeed, adding the vertices of $U$ to $G-U$ one-by-one will preserve the property of being finicyclic.

In turn, it is sufficient to show that there are finitely many cycles of $G$ containing $v$. Every cycle of $G$ that contains $v$ will use a pair of edges incident to $v$, of which there are finitely many, due to local finiteness. So it is sufficient to show that for two given neighbors $u_1$, $u_2$ of $v$, only finitely many cycles contain these three vertices in the order $u_1 v u_2$.

Every such cycle is of the form $u_1 v u_2 P u_1$, where $P$ is a $u_2$--$u_1$ path in $G-v$. By Corollary~\ref{cor:finpath}, there are finitely many such paths, and hence the theorem is proven.
\end{proof}

\subsection{Metric dimension of finicyclic graphs}
In this subsection we return to our original goal of characterizing infinite finicyclic graphs with finite weak metric dimension and finite strong metric dimension. The number of \emph{branch vertices}, which are vertices of degree at least $3$, plays a crucial role in these characterization theorems.

We start by proving a technical lemma, which will be used in the two theorems of this section.

\begin{lemma}\label{lem:partition}
Let $G$ be a finicyclic graph with finitely many branch vertices. Then there is a finite set of vertices $W$ such that $G-W$ has finitely many components, each of which is
\begin{itemize}
\item a ray;
\item a path;
\item or a double ray.
\end{itemize}
Furthermore, if there is a double ray, then that is the only component.
\end{lemma}

\begin{proof}
Let $G$ be a finicyclic graph with finitely many branch vertices. Let $W_1$ be a set of vertices constructed by including a vertex from each cycle, and let $W_2$ the set of branch vertices. Let $W=W_1\cup W_2$. Then $G'=G-W$ is a collection of components of maximum degree at most 2, and $G'$ is cycle-free, so the components of $G'$ are paths, rays, and double rays. Notice that $G'$ has finitely many components: indeed, the removal of a finite set of vertices from a locally finite connected graph can not result in infinitely many components.

It remains to be shown that there may only be one double ray component, and if there is one, there are no other components. Notice that if there is a double ray $D$ component, then $D$ is a double ray in $G$, and for all $v\in V(D)$, we have $\deg_G(v)=2$. This means $D$ is a component of $G$, and being connected, it means $D=G$.
\end{proof}

Slater \cite{Slater},  and Harary and Melter \cite{Harary}, and later Chartrand et al. \cite{Chartrand} (using a new method) characterized the metric dimension of a finite tree:

\begin{theorem}\label{thm:ftreeweak}
If $T$ is a finite tree that is not a path then 
\[ \beta(T) = \sigma(T) - \mbox{ex}\,(T),\]
where $\sigma(T)$ is the number of leaves of $T$,  and $\mbox{ex}\,(T)$  counts the number of branch vertices of $T$ having at least one leaf closer to it than to any other branch vertex.
\end{theorem}

Seb\H{o} and Tannier \cite{Sebo} characterized the strong dimension of a finite tree:

\begin{theorem} \label{ftreestrong}
If $T$ is a non-trivial finite tree then
\[ \beta_{S}(T) = \sigma(T) -1.\]
Moreover, any set of $\sigma(T)-1$  leaves form a strong resolving set for $T$.
\end{theorem}

In the rest of the section, we prove an infinite analogue of Theorem~\ref{thm:ftreeweak} (Theorem~\ref{thm:inftree}) and of Theorem~\ref{ftreestrong} (Theorem~\ref{thm:fintree}).

A special case of Theorem~\ref{thm:inftree} for trees was stated in \cite{Caceres}. In fact, for one of the directions, these authors stated the stronger Proposition~\ref{prop:caceres}. However, their included proof has several errors in it, so we include a (hopefully) correct proof.

We also note that Theorem~\ref{thm:inftree} can be deduced from the already mentioned theorem of C\'aceres et al.~\cite{Caceres} and a theorem of Bir\'o, Novick and Olejnikova \cite{Daniela}, which states that finite changes in a graph can not change the finiteness of the metric dimension. We include a standalone proof, which exposes interesting structural properties of the graphs, and provides insight on the reason why a finite resolving set can not exist in these graphs.

\begin{proposition}\label{prop:caceres}
Let $G$ be a locally finite graph. If $G$ has finitely many branch vertices, then $\beta(G)<\infty$.
\end{proposition}

\begin{proof}
If $G$ is a path, a ray, or double ray, it has a resolving set of size at most $2$. If $G$ is a cycle, it has a resolving set of size $2$.

So in the rest of the proof, we will assume that $G$ is not a path, a ray, a double ray or a cycle. Given that $G$ is connected, we conclude that $G$ has a branch vertex.

Let $W$ be the set of branch vertices and their neighbors; we claim $W$ is a resolving set.

Let $u,v\in V(G)$. If at least one of them is a branch vertex, then the pair is resolved. Otherwise, let $P$ be a shortest $u$--$v$ path.

If $P$ has a branch vertex $w$, then let $w'$ be a neighbor of $w$ on $P$. Note that $w,w'\in W$, and one of them resolves $u$ and $v$.

If $P$ has no branch vertex, then every internal vertex of $P$ is of degree $2$. Let $\mathcal{Q}$ be the set of all paths $Q$ that contain $P$, and for which every internal vertex is of degree $2$. The set $\mathcal{Q}$ is nonempty (it contains $P$) and it is partially ordered by the subpath relation.

Suppose $\mathcal{Q}$ has no maximal element, and let $Q=\bigcup\mathcal{Q}$. Recall that $G$ is not a double ray, so $Q$ is not a double ray. Hence $Q$ is a ray; let its initial vertex be $w$. Recall that $G$ is not a ray, so $\deg(w)\geq 3$, and so $w\in W$. Then $w$ resolves $u$ and $v$.

Now suppose that $\mathcal{Q}$ does have a maximal element, and let it (again) be called $Q$. Note that $Q$ is a path, so it has two end vertices. Consider their degrees.

If the pair of degrees is $(1,1)$, then $G$ is a path. $(1,2)$ is impossible, because the degree $2$ vertex would have to have its neighbors on $Q$. If it is $(2,2)$, then $G$ is a cycle. So at least one end is a branch vertex $w$; without loss of generality, it is the one closer to $v$ than to $u$ on $Q$. Let $w'$ be the neighbor of $w$ on $Q$, and note that $w,w'\in W$.

Suppose $w$ does not resolve $u$ and $v$: $d(w,v)=d(w,u)=\ell$. A $w$--$v$ shortest path must go along $Q$, for otherwise it would have to go through $u$, and so $d(w,u)<d(w,v)$ would follow. Then $d(w',v)=\ell-1$. The length of a $w'$--$u$ shortest path along $Q$ is at least $\ell$, and through $w$, it is at least $\ell+1$. So $d(w',u)\geq\ell$, and hence $w'$ resolves $u$ and $v$.
\end{proof}

\begin{theorem}\label{thm:inftree}
Let $G$ be a finicyclic graph. Then $\beta(G)<\infty$ if and only if $G$ has finitely many branch vertices.
\end{theorem}

\begin{proof}
The backward implication follows directly from Proposition~\ref{prop:caceres}.

For the forward implication, 
suppose $G$ is a finicyclic graph with infinitely many branch vertices. Let the set of rays of $G$ be $\mathcal{R}$. Note that even in a locally finite $G$, $\mathcal{R}$ may be uncountable (for example, see Figure~\ref{fig:binarytree}). We will show that at least one element of $\mathcal{R}$ has infinitely many branch vertices of $G$ on it. Suppose this is not the case. Then, for each $R\in\mathcal{R}$, there exists a vertex $v_R\in V(R)$ such that the subray $v_R R$ contains no branch vertices of $G$. The graph $G'=G-\bigcup_{R\in\mathcal{R}}V(v_RR)$ is connected, and since no branch vertices were removed, it still contains infinitely many branch vertices. In particular, it is an infinite graph, so $G'$ has a ray $R$. Since $R$ is also a ray of $G$, we would have $R\in\mathcal{R}$, which is a contradiction.

Now let $R=v_0v_1\ldots$ be a ray of $G$ with infinitely many branch vertices. Let $W\subseteq V(G)$ be finite. We will  show that $W$ is not a resolving set.  
Pick an index $k$ large enough so that for $i\ge k$, we have
\begin{itemize}
    \item[(a)] $d(v_i,w)> d(w,R)$ for all $w\in W$;
    \item[(b)] $v_i$ has empty intersection with every cycle in $G$.    
\end{itemize}
Now let $\ell$ be the smallest integer with $\ell \ge k$ and $\deg(v_{\ell}) \ge 3$.
Let $v'$ be a vertex adjacent to $v_{\ell}$ other than $v_{\ell-1}$ or $v_{\ell+1}$. Property (b) guarantees that $v'$  is not on $R$. Let $m$ be the largest index for which $d(w,v_m)=d(w,R)$.   Note that it is possible for $w=v_m$.   Also note that, by property (a), $m< \ell$.

We will show that every shortest $w$--$v_{\ell+1}$ path, as well as every shortest $w$--$v'$ path,  goes through $v_{\ell}$, from which it will follow that $w$ cannot resolve $v_{\ell+1}$ and $v'$.    Indeed, let $P'$ be a shortest $w$--$v_m$ path. The interior of $P'$ has empty intersection with $R$. Hence we may define a path $P$ as 
$$P=wP'v_mRv_{\ell+1}.$$
(Note that in the case that $w=v_m$ we have
$P=v_mRv_{\ell+1}$.)
Assume by way of contradiction that $Q$ is a shortest $w$--$v_{\ell+1}$ path not containing $v_{\ell}$.  It would then follow immediately from Lemma \ref{lem:cyclecreator} that $v_{\ell}$ is contained in a cycle,  contradicting condition (b).   To make an analogous argument for any shortest $w$--$v'$  path, it remains to show that $v'$  is not contained in $P'$.  Indeed, if we assume for a contradiction that $v'\in P'$, then $v'P'v_mRv_{\ell}v'$ would be a cycle containing $v_{\ell}$,  contradicting property (b). 
\end{proof}

Of course, if $\beta(G)=\infty$, then $\beta_S(G)=\infty$, so we only need to resolve the characterization for strong dimension when there are finitely branch vertices.
This case will also depend on the number of ends in $G$. By Corollary~\ref{cor:3timer}, if $G$ has more than one end, then it has infinite strong dimension. The following theorem resolves the remaining case.
\begin{figure}
    \centering
    \begin{tikzpicture}
        \node[circle, fill=black,scale = 0.5,label] (A) at (0,0) {};
        \node[circle, fill=black,scale = 0.5,label] (B) at (-1,-.5) {};
        \node[circle, fill=black,scale = 0.5,label] (C) at (1,-.5) {};
        \node[circle, fill=black,scale = 0.5,label] (D) at (-1.5,-1) {};
        \node[circle, fill=black,scale = 0.5,label] (E) at (-.5,-1) {};
        \node[circle, fill=black,scale = 0.5,label] (F) at (1.5,-1) {};
        \node[circle, fill=black,scale = 0.5,label] (G) at (.5,-1) {};

        \node[circle, fill=black,scale = 0.5,label] (h) at (-1.7,-1.5) {};
        \node[circle, fill=black,scale = 0.5,label] (i) at (-.3,-1.5) {};
        \node[circle, fill=black,scale = 0.5,label] (j) at (1.7,-1.5) {};
        \node[circle, fill=black,scale = 0.5,label] (k) at (.3,-1.5) {};

        \node[circle, fill=black,scale = 0.5,label] (l) at (-1.3,-1.5) {};
        \node[circle, fill=black,scale = 0.5,label] (m) at (-.7,-1.5) {};
        \node[circle, fill=black,scale = 0.5,label] (n) at (1.3,-1.5) {};
        \node[circle, fill=black,scale = 0.5,label] (o) at (.7,-1.5) {};
        \node at (-1,-2) {\rotatebox{90}{$\cdots$}};
        \node at (1,-2) {\rotatebox{90}{$\cdots$}};
        
        \draw (A) to (B);
        \draw (A) to (C);
        \draw (D) to (B);
        \draw (E) to (B);
        \draw (F) to (C);
        \draw (G) to (C);
        \draw (D) to (l);
        \draw (D) to (h);
        \draw (E) to (m);
        \draw (E) to (i);
        \draw (F) to (n);
        \draw (F) to (j);
        \draw (G) to (k);
        \draw (G) to (o);
    \end{tikzpicture}
    \caption{The infinite binary tree has an uncountable number of rays. Each ray can be described by the binary representation of any number in $[0,1]$.}
    \label{fig:binarytree}
\end{figure}
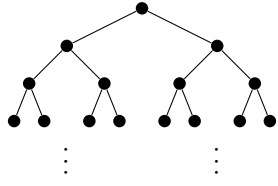
\begin{theorem}\label{thm:fintree}
Let $G$ be a finicyclic graph with one end. If $G$ has finitely many branch vertices, then $\beta_S(G)<\infty$.
\end{theorem}

\begin{proof}
Let $G$ be a finicyclic graph with one end, and with finitely many branch vertices. We apply Lemma~\ref{lem:partition} to get a finite set $W$ and a partition $D,R_1,\ldots,R_k,P_1,\ldots,P_\ell$. Notice that $G$ is not a double-ray, because it has only one end, so $D=\emptyset$.

Since $G$ is infinite, we have $k\geq 1$. In fact, we can show that $k=1$. To see this, suppose that $k\geq 2$. Then $G[R_1]$ and $G[R_2]$ are disjoint rays of $G$, and for all $v\in R_1\cup R_2$, we have $\deg_G(v)\leq 2$. So every $R_1$--$R_2$ path in $G$ goes through $W$. This shows that they are in different ends, contradicting the assumption that $G$ has one end.

Now $r_1$ be the initial vertex of $R_1$, and let $W'=W\cup P_1\cup\cdots\cup P_\ell\cup\{r_1\}$. Note that $0<|W'|<\infty$. We will prove that $W'$ is a resolving set of $G$. Let $u,v\in V(G)$, and assume that $u,v\not\in W'$. Then $u,v\in R_1$. Without loss of generality, $d(u,r_1)<d(v,r_1)$ (they can not be equal), so then $u$ is on a shortest $v$--$r_1$ path in $G$.
\end{proof}

We summarize the findings of this section in the following corollary.

\begin{corollary}
Let $G$ be an infinite finicyclic graph.
\begin{itemize}
\item The weak dimension of $G$ is finite if and only if $G$ has finitely many branch vertices.
\item The strong dimension of $G$ is finite if and only if $G$ has one end and finitely many branch vertices.
\end{itemize}
\end{corollary}

\section{Two open questions on strong dimension}

Another possible approach to characterizing infinite graphs with finite strong dimension involves considering a specific partial orientation of these graphs.

Let $G$ be a graph and $W\in V(G)$. Rays that have their initial vertex in $W$ are called \emph{$W$-rays}. A \emph{geodesic ray} is a ray $R$ such that for all $u,v\in V(R)$, the path $uRv$ is a shortest $u$--$v$ path in $G$. Accordingly, \emph{geodesic $W$-rays} are geodesic rays with their initial vertex in $W$.

For a ray $R$ with vertices (in order outward) $v_0,v_1,\ldots$, we can definite a natural orientation with $v_i\to v_{i+1}$ for all $i\in\mathbb{N}$; we will use the notation $\d{R}$ for this natural orientation.

We will use the term \emph{directed graph} for a graph in which some edges are directed, but others are not. One can view the undirected edges as directed in both ways. An \emph{orientation} of $G$ is a directed graph formed from $G$ with some of its edges directed.

\begin{definition}
Let $G$ be a graph and $W\subseteq V(G)$. Let $\mathcal{W}$ be the collection of the geodesic $W$-rays. Let $G_W$ be an orientation of $G$ defined as follows. For the edge $uv\in E(G)$, we orient $u\to v$ in $G_W$, if there is an $R\in\mathcal{W}$ for which $u,v\in V(R)$, and for all $R\in\mathcal{W}$ for which $u,v\in V(R)$, the orientation in $\d{R}$ is $u\to v$; otherwise we leave it undirected.
\end{definition}

In a directed graph, we say there is $u\to v$ path, if there is a $u$--$v$ path that traverses every directed edge in the correct direction. Undirected edges may be traversed in either direction.

\begin{definition}
A \emph{bad pair} in $G_W$ is a pair of vertices $u$, $v$, such that there is neither a $u\to v$ path, nor a $v\to u$ path.
\end{definition}
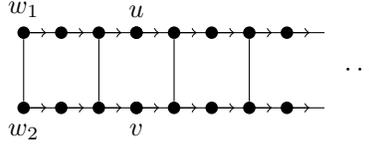
\begin{figure}
    \centering
    \begin{tikzpicture}
        \tikzset{edge/.style = {->}}
        
        \node[circle, fill=black,scale = 0.5,label = above:$w_1$] (A) at (0,0) {};
        \node[circle, fill=black,scale = 0.5] (B) at (1,0) {};
        \node[circle, fill=black,scale = 0.5] (C) at (2,0) {};
        \node[circle, fill=black,scale = 0.5,label = above:$u$] at (1.5,0) {};
        \node[circle, fill=black,scale = 0.5,label = below:$v$] at (1.5,-1) {};
        \node[circle, fill=black,scale = 0.5,label = below:$w_2$] (D) at (0,-1) {};
        \node[circle, fill=black,scale = 0.5] (E) at (1,-1) {};
        \node[circle, fill=black,scale = 0.5] (F) at (2,-1) {};
        \node[circle, fill=black,scale = 0.5] (H) at (3,0) {};
        \node[circle, fill=black,scale = 0.5] (I) at (3,-1) {};

        \node (G) at (4.5,-.5) {$\cdots$};

        \draw (A) -- (4,0);
        \draw (D) -- (4,-1);
        \draw (A) -- (D);
        \draw (C) -- (F);
        \draw (B) -- (E);
        \draw (H) -- (I);

        \foreach \i in {1,...,4}
        {
        \node[circle, fill=black,scale = 0.5] at (\i-0.5,0) {};
        \draw[edge] (\i-1,0) -- (\i-0.7,0);
        \draw[edge] (\i-0.5,0) -- (\i-.2,0);
        \node[circle, fill=black,scale = 0.5] at (\i-0.5,-1) {};
        \draw[edge] (\i-1,-1) -- (\i-0.7,-1);
        \draw[edge] (\i-0.5,-1) -- (\i-.2,-1);
        }
    \end{tikzpicture}
    \caption{The vertices $u$, $v$ are a bad pair in $G_W$ when $G$ is the infinite broken ladder and $W=\{w_1,w_2\}$.}
    \label{fig:badpair}
\end{figure}

The two open questions we offer are the following.

\begin{question}
Let $G$ be a graph. Is it true that if for all nonempty finite $W\in V(G)$, there is a bad pair in $G_W$, then $\beta_S(G)=\infty$?
\end{question}

\begin{question}
Let $G$ be a graph. Is it true that if for all finite $W\in V(G)$, there is no bad pair in $G_W$, then $\beta_S(G)<\infty$?
\end{question}

\bibliographystyle{abbrv}
\bibliography{smetric}

\end{document}